\documentclass[11pt]{article}

\usepackage{amsmath,amssymb,amsthm,mathtools}
\usepackage[T1]{fontenc}
\usepackage{authblk}
\usepackage{enumitem}

\usepackage{tikz}
\usetikzlibrary{positioning,arrows.meta}

\usepackage{float}
\usepackage{placeins}
\usepackage{graphicx}

\title{\bf Celebrating the Day of $\pi$: Joyful Variations on Euler's Identity}

\author{\Large Takao Inou\'{e}}
\affil{\large Faculty of Informatics, Yamato University, \\ Osaka, Japan\footnote{Email: inoue.takao@yamato-u.ac.jp; \\ Personal Email: takaoapple@gmail.com \\ [I prefer my personal email address for correspondence.]}}
\date{March 14, 2026}

\newtheorem{definition}{Definition}[section]

\newtheorem{proposition}[definition]{Proposition}

\begin{document}

\maketitle

\begin{abstract}
This short essay celebrates the mathematical meaning of Pi Day through Euler's formula
\[
e^{ix}=\cos x+i\sin x,
\]
from which Euler's identity
\[
e^{i\pi}+1=0
\]
follows immediately. We briefly note the historical background of the formula, usually traced to Euler's \emph{Introductio in analysin infinitorum} (1748), while also mentioning Roger Cotes's earlier precursor of 1714. We compare Euler's identity, in an explicitly analogical way, with several famous formulas in physics in order to highlight its remarkable compactness and conceptual richness. We then consider a number of joyful variations arising from the same Eulerian source, including the negative-angle case, prime-number multiples, the substitution $x=\pi/2$, and a functional-equation variation of the form
\[
f(i\pi x)+1=0.
\]
This last variation leads naturally to a contrast between rigidity in the holomorphic setting and freedom in the discrete interpolation setting. The central aim is to organize these observations into two simple families of variations: geometric-angle variations and functional-equation variations.

The earlier part of the exposition is intended to be accessible to motivated high-school students, while the later discussion points toward more advanced ideas from complex analysis.
\end{abstract}


\medskip


\tableofcontents

\section{Introduction}

March 14 is celebrated as the Day of $\pi$. For many students, $\pi$ first appears in geometry as the number associated with circles. One learns formulas such as the circumference $2\pi r$ and the area $\pi r^2$, and so $\pi$ first enters mathematics through familiar geometric pictures. Yet Pi Day can also be an occasion to remember that $\pi$ appears far beyond elementary geometry. In particular, it appears in one of the most famous formulas in all of mathematics, namely Euler's formula
\[
e^{ix}=\cos x+i\sin x.
\]

Historically, this formula is generally associated with Leonhard Euler and is usually traced to his \emph{Introductio in analysin infinitorum} of 1748 \cite{Euler1748,EulerBlanton,Calinger2016}. At the same time, historians of mathematics have often noted that Roger Cotes had already obtained an important precursor in 1714, expressed in logarithmic form rather than in the now standard exponential form \cite{Cotes1714,Stillwell2010,Maor1994}. Thus, from a bibliographical and historical point of view, it is reasonable to say that the standard formula is Euler's, while part of its conceptual background goes back to Cotes.

For the purposes of this essay, the main point is simple: Euler's identity
\[
e^{i\pi}+1=0
\]
is not an isolated miracle. It comes directly from Euler's formula. Indeed, if we substitute $x=\pi$, we obtain
\[
e^{i\pi}=\cos \pi+i\sin \pi.
\]
Since
\[
\cos \pi=-1
\qquad\text{and}\qquad
\sin \pi=0,
\]
it follows that
\[
e^{i\pi}=-1,
\]
and hence
\[
e^{i\pi}+1=0.
\]
Thus the famous identity is simply a special case of Euler's formula, obtained by choosing the angle $x=\pi$.

To celebrate Pi Day in a more playful spirit, I also prepared a small plate inspired by the formula $e^{i\pi}+1=0$, shown in Figure~\ref{fig:pi-day-plate}.

\begin{figure}[H]
\centering
\includegraphics[width=0.72\textwidth]{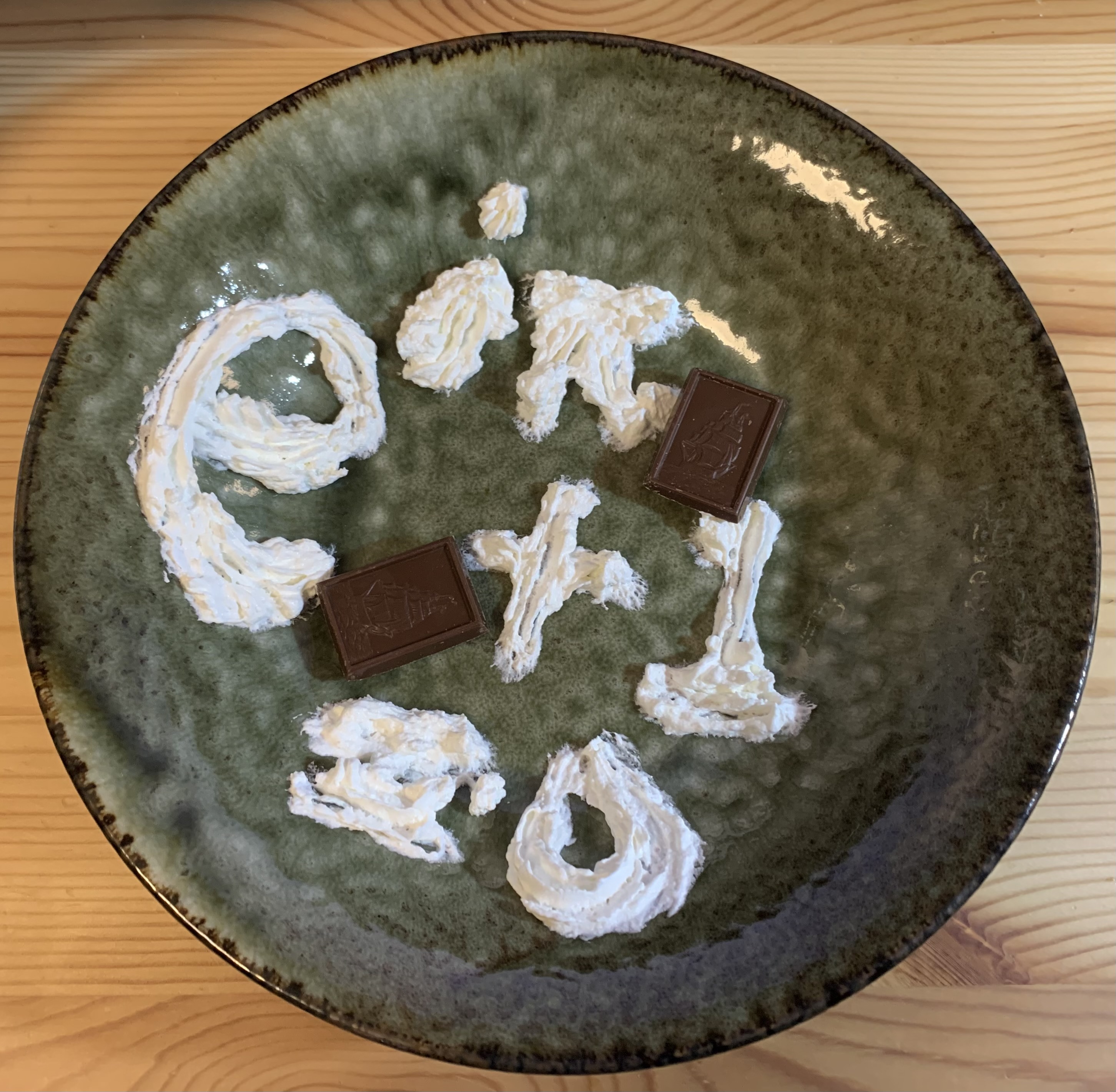}
\caption{A joyful Pi Day plate prepared by the author, inspired by Euler's identity $e^{i\pi}+1=0$.}
\label{fig:pi-day-plate}
\end{figure}

Why is this identity so often regarded as exceptional? One reason is that it brings together
\[
\pi,\qquad e,\qquad i,\qquad 1,\qquad 0
\]
in a single equation. These symbols come from very different parts of mathematics. The number $\pi$ comes from geometry and circles. The constant $e$ is the base of natural logarithms and belongs to the theory of growth, logarithms, and exponential functions. The symbol $i$ is the imaginary unit and belongs to the extension of the number system into the complex domain. The numbers $1$ and $0$ are among the most basic constants in arithmetic and algebra. Yet all of them appear together in one short formula.

It may be helpful to appreciate this fact by analogy with some famous formulas in physics. Hooke's law,
\[
f=-kx,
\]
contains two principal physical quantities, force and displacement, together with the constant $k$. Newton's equation,
\[
f=ma,
\]
connects three central physical quantities: force, mass, and acceleration. Einstein's well-known formula,
\[
E=mc^2,
\]
also connects three major physical notions: energy, mass, and the speed of light.

Of course, this comparison is only an analogy, and it should not be overstated. It is not a completely strict comparison of like with like. In the formulas of Hooke, Newton, and Einstein, one is mainly counting physical quantities, whereas in Euler's identity one is pointing to numbers, constants, and symbols. Thus, the comparison is not intended as a rigorous scientific classification. Rather, it is an illustrative analogy whose purpose is to highlight how unusually compact and conceptually rich Euler's identity is.

Seen in this light, Pi Day is not only a celebration of the decimal number $3.14$, and not only a celebration of circles. It is also a good opportunity to appreciate a deeper unity in mathematics. Euler's formula shows that exponential functions, trigonometric functions, and complex numbers are not separate topics, but are profoundly connected. Euler's identity, obtained from that formula by the single substitution $x=\pi$, is then a particularly striking instance of this unity.

For historical background on Euler's mathematical work more broadly, see \cite{BradleySandifer2007,Dunham1999,Calinger2016}.

At the same time, the classical identity is not the end of the story. Once one takes Euler's formula as the true source, one can begin to form natural variations by changing the angle. Some of these preserve the same basic pattern, while others shift the resulting value from the real axis to the imaginary axis. Still others lead to small functional-equation problems whose behavior depends on whether one works continuously or discretely. In this sense, Euler's identity may be read not only as a final destination, but also as the opening theme of a small mathematical set of variations.

The aim of the present note is to make this point in a focused way. After the introductory discussion, the variations are organized into two broad families: geometric-angle variations obtained by changing the input of Euler's formula, and functional-equation variations obtained by prescribing values along the imaginary axis. This two-part organization is intended to make the mathematical theme of the paper more explicit.

For this reason, the Day of $\pi$ can be celebrated not merely as a playful date on the calendar, but also as a reminder of one of the clearest examples of mathematical connection and elegance. The formula
\[
e^{i\pi}+1=0
\]
is short enough to be written on a single line, yet rich enough to bring together several of the most important ideas in mathematics. That is why it continues to fascinate students, teachers, and mathematicians alike.

The exposition is intentionally written in a layered way: the earlier sections are meant to be readable even for motivated high-school students, while the later sections gradually lead toward more advanced viewpoints from complex analysis.

\section{Geometric-Angle Variations on Euler's Identity}

The introduction has focused on the classical identity
\[
e^{i\pi}+1=0,
\]
derived from Euler's formula
\[
e^{ix}=\cos x+i\sin x.
\]
From this point, however, one may begin a number of natural variations. These variations are not meant to replace the classical identity, but rather to show how flexible and suggestive Euler's formula is. The organizing principle in the present section is simple: we vary the angle and observe how the corresponding endpoint in the complex plane changes.

\subsection{The negative-angle variation}

A first variation is obtained by replacing $\pi$ with $-\pi$. Then Euler's formula gives
\[
e^{-i\pi}=\cos(-\pi)+i\sin(-\pi).
\]
Since cosine is an even function and sine is an odd function, we have
\[
\cos(-\pi)=\cos(\pi)=-1,
\qquad
\sin(-\pi)=-\sin(\pi)=0.
\]
Hence
\[
e^{-i\pi}=-1,
\]
and therefore
\[
e^{-i\pi}+1=0.
\]

This identity is numerically the same as the classical one, but conceptually it adds a new feature: the explicit appearance of a negative sign in the exponent. If one wishes to read Euler's classical identity as a unification of the five symbols
\[
\pi,\qquad e,\qquad i,\qquad 1,\qquad 0,
\]
then the formula
\[
e^{-i\pi}+1=0
\]
may be viewed as a variation in which the idea of negativity is also explicitly present. In that suggestive sense, one may speak of a union of six concepts rather than five. Of course, this is an interpretive remark rather than a strict mathematical count, since the minus sign is an operation rather than a number on the same footing as the others.

\subsection{A prime-number variation}

One may next ask what happens if $\pi$ is replaced by $\pi p$, where $p$ is a prime number. Then
\[
e^{-i\pi p}
 = \cos(\pi p)-i\sin(\pi p).
\]
If $p$ is an odd prime, then $p$ is an odd integer, and so
\[
\cos(\pi p)=-1,
\qquad
\sin(\pi p)=0.
\]
Therefore
\[
e^{-i\pi p}=-1
\qquad\text{for every odd prime }p.
\]
It follows that
\[
e^{-i\pi p}+1=0
\qquad\text{for every odd prime }p.
\]

This gives a genuine prime-number variation of Euler's identity. It shows that the same formal pattern reappears for all odd primes. In this sense, the identity is not tied only to the single angle $\pi$, but extends periodically along odd integral multiples of $\pi$.

It is worth noting, however, that the formula
\[
e^{-i\pi p}+p=0
\]
does \emph{not} hold in general. Indeed, for every odd prime $p$ one has
\[
e^{-i\pi p}=-1,
\]
so that
\[
e^{-i\pi p}+p=p-1,
\]
which is not zero. Thus the correct prime-number analogue is
\[
e^{-i\pi p}+1=0
\qquad (p\text{ an odd prime}).
\]

For the exceptional prime $p=2$, one has
\[
e^{-2i\pi}=\cos(2\pi)-i\sin(2\pi)=1,
\]
so that
\[
e^{-2i\pi}+1=2.
\]
Thus the prime $2$ behaves differently from the odd primes, as one would expect from its special parity.

\subsection{The quarter-turn variation: replacing $\pi$ by $\pi/2$}

Another natural variation arises by replacing $\pi$ with $\pi/2$. Then Euler's formula gives
\[
e^{i\pi/2}=\cos(\pi/2)+i\sin(\pi/2)=i,
\]
and similarly
\[
e^{-i\pi/2}=\cos(-\pi/2)+i\sin(-\pi/2)=-i.
\]
Hence we obtain the identities
\[
e^{i\pi/2}-i=0
\qquad\text{and}\qquad
 e^{-i\pi/2}+i=0.
\]

These formulas are different in character from Euler's identity
\[
e^{i\pi}+1=0.
\]
In the classical case, the exponential expression lands on the real number $-1$, and so the equation closes with the basic real numbers $1$ and $0$. By contrast, when the angle is $\pi/2$, the exponential expression lands on the imaginary unit itself. Thus the result is not an identity involving $1$ and $0$ alone, but one in which $i$ remains visible on the surface.

Geometrically, this is natural. Multiplication by $e^{ix}$ corresponds to rotation in the complex plane by angle $x$. The angle $\pi$ represents a half-turn, sending $1$ to $-1$, while the angle $\pi/2$ represents a quarter-turn, sending $1$ to $i$. Thus the passage from $\pi$ to $\pi/2$ changes the endpoint from the negative real axis to the positive imaginary axis.

One may therefore say that the substitution $x=\pi$ produces a particularly balanced identity between exponentials, geometry, and arithmetic, whereas the substitution $x=\pi/2$ produces a more visibly complex-valued variation. Both arise naturally from the same Eulerian source, but they highlight different aspects of the formula.

\subsection{A general perspective}

These examples show that Euler's identity is only one distinguished member of a wider family of relations generated by Euler's formula. The classical equation
\[
e^{i\pi}+1=0
\]
remains special because it joins together several of the most fundamental constants and symbols of mathematics in an exceptionally compact form. Yet the nearby formulas
\[
e^{-i\pi}+1=0,\quad
e^{-i\pi p}+1=0\ (p\text{ odd prime}),\quad
e^{-i\pi/2}+i=0, \quad
e^{i\pi/2}-i=0
\]
show that this famous identity also invites systematic variation.

In this sense, Euler's formula is not merely the source of a single celebrated equation. It is the source of a small family of related mathematical variations.

\begin{figure}[H]
\centering
\includegraphics[width=0.82\textwidth]{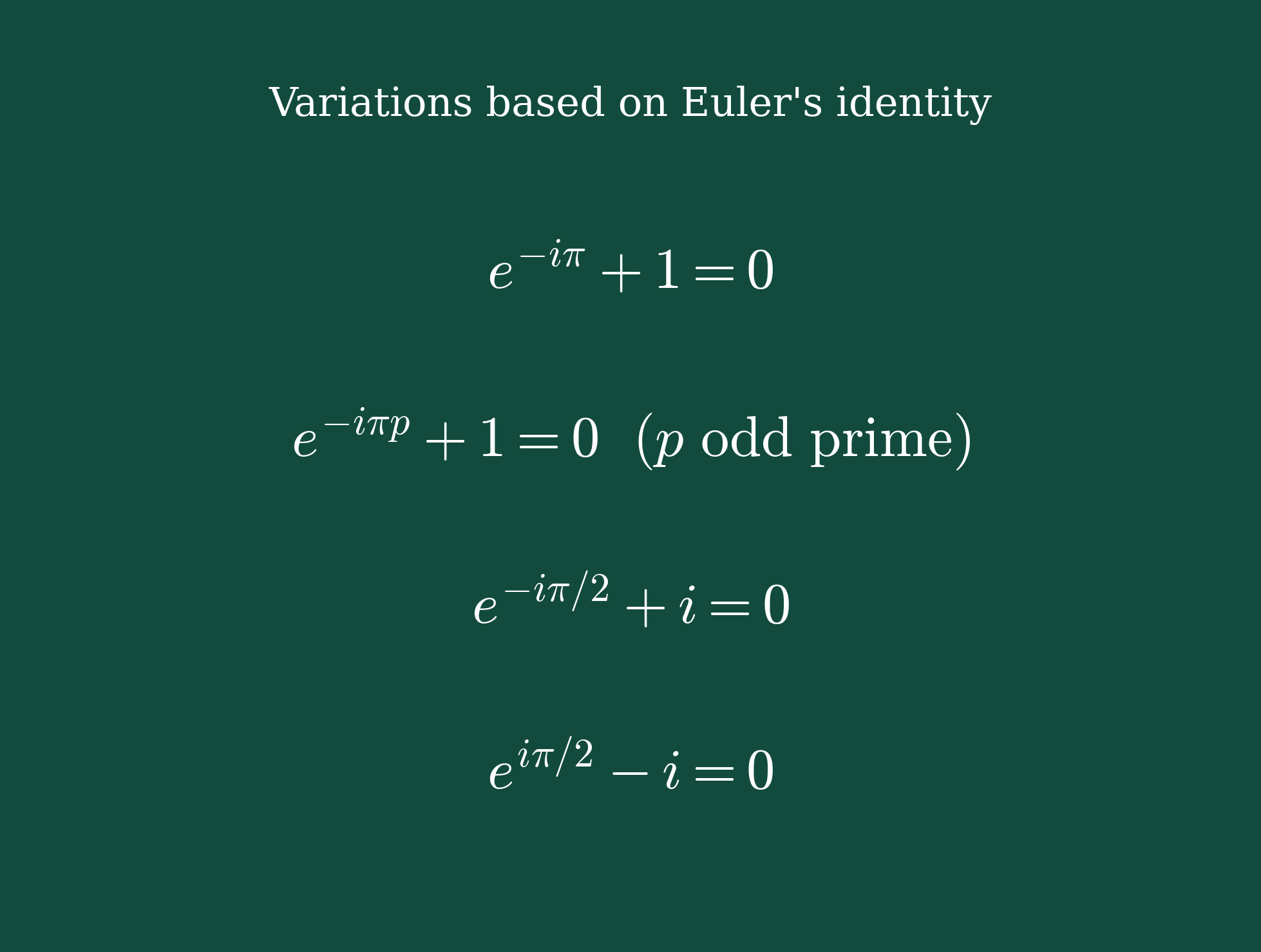}
\caption{Four joyful variations on Euler's identity.}
\label{fig:pi-variations-green}
\end{figure}

\section{A Functional-Equation Variation}

The previous section considered several explicit variations obtained by substituting special angles into Euler's formula. One may also move in a different direction and ask for a function $f$ satisfying
\[
f(i\pi x)+1=0.
\]
Equivalently,
\[
f(i\pi x)=-1.
\]
At first sight, this may look like a new kind of Eulerian functional equation. However, its mathematical nature depends strongly on what set of values of $x$ is allowed, and on what regularity is imposed on the function $f$. The organizing principle of the present section is the contrast between holomorphic rigidity on continuous sets and interpolation freedom on discrete sets.

\subsection{The continuous case}

Suppose first that the relation
\[
f(i\pi x)=-1
\qquad (x\in\mathbb R)
\]
is required for all real numbers $x$. Then $i\pi x$ runs through the entire imaginary axis. In other words, the condition says that $f$ takes the constant value $-1$ on the whole imaginary axis.

Now assume that $f$ is holomorphic on a connected open set containing the imaginary axis. Consider the function
\[
g(z):=f(z)+1.
\]
Then $g$ is holomorphic, and the condition above implies that
\[
g(i\pi x)=0
\qquad (x\in\mathbb R).
\]
Thus $g$ vanishes on an infinite set having accumulation points in its domain. By the identity theorem for holomorphic functions, it follows that
\[
g(z)\equiv 0,
\]
and hence
\[
f(z)\equiv -1.
\]
So in the holomorphic setting, the functional equation has only the trivial solution. This shows that, in the continuous case, the problem is actually too rigid to produce an interesting family of analytic solutions \cite{Conway1978,SteinShakarchi2003}.

\begin{proposition}
Let $U\subseteq \mathbb C$ be a connected open set containing the imaginary axis. If $f:U\to\mathbb C$ is holomorphic and satisfies
\[
f(i\pi x)=-1\qquad (x\in\mathbb R),
\]
then $f\equiv -1$ on $U$.
\end{proposition}

\begin{proof}
Apply the identity theorem to $g(z)=f(z)+1$, which vanishes on the set $\{i\pi x:x\in\mathbb R\}$.
\end{proof}

\subsection{The unrestricted case}

If no regularity is imposed on $f$, the picture changes completely. In that case, one may define $f$ arbitrarily away from the imaginary axis, while requiring only that
\[
f(iy)=-1
\]
for all real $y$. Hence there are infinitely many such functions. From this point of view, the equation is no longer a rigid analytic condition, but rather a prescribed-value condition on a particular subset of the complex plane.

Thus the same formal expression
\[
f(i\pi x)+1=0
\]
has two very different meanings. In the holomorphic category it forces $f$ to be constant, while in the unrestricted category it leaves enormous freedom.

\subsection{The discrete case and interpolation}

A more interesting variation arises when the parameter $x$ is restricted to a discrete set, for example the integers:
\[
f(i\pi n)=-1
\qquad (n\in\mathbb Z).
\]
Now the condition is imposed only on a discrete subset of the imaginary axis. In this setting, one is naturally led to an interpolation problem: can one construct a holomorphic or entire function taking prescribed values on a discrete set?

This is a classical direction in complex analysis. Once the points are discrete, the condition no longer forces constancy by the identity theorem, because there is no accumulation point in the finite plane. Instead, one enters the realm of interpolation theory for analytic or entire functions. In this sense, the discrete version is mathematically much richer than the continuous one \cite{Boas1954,Levin1996,Conway1978}.

In fact, this is not merely an existence question: one can write down an explicit nonconstant entire solution. Namely,
\[
f(z)=-1+\sinh^2 z
\]
satisfies
\[
f(i\pi n)=-1
\qquad (n\in\mathbb Z),
\]
because
\[
\sinh(i\pi n)=i\sin(\pi n)=0.
\]
Thus the discrete problem admits concrete nontrivial solutions.

The same observation applies if one restricts to odd primes:
\[
f(i\pi p)=-1
\qquad (p\ \text{odd prime}).
\]
Again, this becomes a discrete interpolation problem rather than a rigidity statement. One may then ask about existence, explicit construction, growth, or uniqueness under additional constraints.

\begin{proposition}
For any prescribed constant $c\in\mathbb C$, the family
\[
f(z)=c+a\sinh^2 z\qquad (a\in\mathbb C)
\]
consists of entire functions satisfying
\[
f(i\pi n)=c\qquad (n\in\mathbb Z).
\]
More generally, every function of the form
\[
f(z)=c+\sinh z\,g(z),
\]
with $g$ entire, satisfies the same condition.
\end{proposition}

\begin{proof}
Since $\sinh(i\pi n)=i\sin(\pi n)=0$ for every integer $n$, both formulas give $f(i\pi n)=c$.
\end{proof}

\subsection{A further value-prescription variation: the case $f(i\pi x)=1/2$}

A closely related variation is obtained by replacing the value $-1$ with the value $1/2$. Thus we ask for a function $f$ satisfying
\[
f(i\pi x)=\frac12.
\]
Equivalently,
\[
f(i\pi x)-\frac12=0.
\]

The mathematical structure of this problem is essentially the same as in the previous case. If the relation is required for all real $x$, then $i\pi x$ runs through the whole imaginary axis, so the condition says that $f$ is equal to the constant $1/2$ on that axis.

Now assume that $f$ is holomorphic on a connected open set containing the imaginary axis, and define
\[
g(z):=f(z)-\frac12.
\]
Then $g$ is holomorphic, and
\[
g(i\pi x)=0
\qquad (x\in\mathbb R).
\]
Hence $g$ vanishes on a set having accumulation points in its domain. By the identity theorem for holomorphic functions, it follows that
\[
g(z)\equiv 0,
\]
and therefore
\[
f(z)\equiv \frac12.
\]
So, just as in the case $f(i\pi x)=-1$, the continuous holomorphic version is completely rigid.

On the other hand, if no regularity is imposed on $f$, then there are infinitely many such functions: one may simply define $f$ to be $1/2$ on the imaginary axis and choose its values arbitrarily elsewhere. Thus, in the unrestricted setting, the condition becomes only a prescribed-value condition.

As before, the most interesting case is the discrete one. For example, one may ask for an entire function satisfying
\[
f(i\pi n)=\frac12
\qquad (n\in\mathbb Z).
\]
In fact, one can again write down an explicit nonconstant entire solution:
\[
f(z)=\frac12+\sinh^2 z.
\]
Indeed,
\[
\sinh(i\pi n)=i\sin(\pi n)=0,
\]
so that
\[
f(i\pi n)=\frac12+\sinh^2(i\pi n)=\frac12.
\]

Similarly, one may consider
\[
f(i\pi p)=\frac12
\qquad (p\ \text{odd prime}),
\]
which is again a discrete interpolation problem.

This variation is conceptually instructive. The value $-1$ is distinguished because it is exactly the value produced by Euler's formula at the angle $\pi$. By contrast, the value $1/2$ has no such immediate Eulerian privilege. Thus the equation
\[
f(i\pi x)=\frac12
\]
shows more clearly that the underlying issue is not the special arithmetic of Euler's identity alone, but the broader analytic contrast between rigidity on continuous sets and freedom on discrete sets.

\subsection{A natural conclusion}

Thus the expression
\[
f(i\pi x)+1=0
\]
does not seem to define a famous special functional equation in its own right. Rather, it opens two contrasting directions. If the variable $x$ ranges continuously and $f$ is holomorphic, the equation collapses to the trivial constant solution
\[
f\equiv -1.
\]
If, on the other hand, the variable is restricted to a discrete set such as the integers or the odd primes, the problem becomes a natural interpolation question in complex analysis, and even explicit nonconstant entire solutions can be exhibited.

The same pattern persists when one replaces the value $-1$ by another prescribed constant such as $1/2$. In the continuous holomorphic setting, one again obtains only the constant solution, whereas in the discrete setting one is led to interpolation problems and explicit examples.

In this way, the functional-equation variation continues the theme of the previous sections. Euler's identity suggests not only isolated formulas, but also broader families of problems in which rigidity and freedom coexist in an instructive balance.

\bigskip

\noindent Takao Inou\'{e}

\noindent Faculty of Informatics

\noindent Yamato University

\noindent Katayama-cho 2-5-1, Suita, Osaka, 564-0082, Japan

\noindent inoue.takao@yamato-u.ac.jp
 
\noindent (Personal) takaoapple@gmail.com (I prefer my personal mail)
\bigskip

\end{document}